\documentclass[reqno,11pt,a4paper]{amsart}

\usepackage{amsmath, amsthm, amsfonts, amssymb}

\usepackage[applemac]{inputenc}

\usepackage{mathtools}
\mathtoolsset{showonlyrefs}

\usepackage[parfill]{parskip}    
\theoremstyle{plain}
\newtheorem{theorem}{Theorem}
\newtheorem{proposition}[theorem]{Proposition}
\newtheorem{lemma}[theorem]{Lemma}

\theoremstyle{definition}

\theoremstyle{remark}

\def\R{\mathbb{R}}	
 		
\renewcommand{\geq}{\geqslant}

\let\ker\relax
\DeclareMathOperator{\ker}{Ker}
\def\cA{\mathcal{A}}
\def\hcA{\hat{\cA}}
\def\cF{\mathcal{F}}
\def\hcF{\hat{\cF}}

\def\d{\partial}

\begin{document}

\title[Bi-Hamiltonian cohomology, scalar case]{The bi-Hamiltonian cohomology of a scalar Poisson pencil}
\author{Guido Carlet}
\author{Hessel Posthuma}
\author{Sergey Shadrin}
\address{Korteweg-de Vries Instituut voor Wiskunde, 
Universiteit van Amsterdam, Postbus 94248,
1090GE Amsterdam, Nederland}
\email{g.carlet@uva.nl, h.b.posthuma@uva.nl, s.shadrin@uva.nl}
\begin{abstract}
We compute the bi-Hamiltonian cohomology of an arbitrary dispersionless Poisson pencil in a single dependent variable using a spectral sequence method. As in the KdV case, we obtain that $BH^p_d(\hcF, d_1,d_2)$ is isomorphic to $\R$ for $(p,d)=(0,0)$, to $C^\infty (\R)$ for $(p,d)=(1,1)$, $(2,1)$, $(2,3)$, $(3,3)$, and vanishes otherwise.
\end{abstract}
\maketitle


\section{Introduction}

We consider evolutionary integrable systems of PDEs in one spatial variable that possess a bi-Hamiltonian structure, i.e., a compatible pair of Poisson brackets on the space of local functionals. One of the main tools in the classification of this kind of integrable systems, up to the action of the Miura group, is the deformation theory of the dispersionless limit of the pencil of Poisson structures. Under the extra assumption of semi-simplicity, the deformation theory of a pencil of Poisson structures is controlled by their bi-Hamiltonian cohomology. This concept was introduced in~\cite{dz01,lz05,dlz06}, and further intensively studied in the last decade. We refer the reader also to~\cite{al11,al12,b08,cps14,cps15, lz11, lz13,l02} for different aspects of the definition, usage, and computation of the bi-Hamiltonian cohomology. 

In particular, for semi-simple pencils of hydrodynamic Poisson structures, the deformation space is computed in~\cite{dlz06}, where the corresponding functional parameters are called central invariants, and the obstruction space is shown to be equal to zero in~\cite{cps15}. (For the case of the dispersionless KdV Poisson pencil the vanishing of the obstructions was proved in the paper~\cite{lz13}, in which important developments of the formalism were introduced, see also~\cite{lz11}). This implies that for a given semi-simple Poisson pencil of hydrodynamic type and for a given set of central invariants there exist a dispersive deformation, which is unique up to the action of the Miura group. 

Despite many recent developments in our understanding of bi-Hamiltonian cohomology, the full description of the bi-Ha\-mil\-tonian cohomology is known only for the case of the Poisson pencil of the dispersionless KdV hierarchy (it was conjectured in~\cite{lz13} and proved in~\cite{cps14}). In particular, it is clear that the spectral sequences constructed in general case in~\cite{cps15} are not suitable for this purpose, and new computational ideas are required. Especially disappointing here is that, in the general semi-simple case, there is no uniform computation of the space of deformations and the space of obstructions: these two spaces are computed in~\cite{dlz06} and in~\cite{cps15} via completely different methods. 

In this paper we study the case of a hydrodynamic Poisson pencil in one dependent variable, and we compute at once its full bi-Hamiltonian cohomology. Since this paper is more of technical nature, we assume the reader to be already familiar with the standard notations and we refer to~\cite{dz01} and~\cite{lz13} for an extended introduction to this topic. 

\subsection{Basic notions}

A Poisson bracket of hydrodynamic type in a single dependent variable is of the form
\begin{equation}
\{ u(x) , u(y) \} = g(u(x)) \delta'(x-y) + \frac12 g'(u(x)) u_x(x) \delta(x-y) 
\end{equation}
where $g(u)$ is a smooth non-vanishing function on $\R$.

A pair of compatible Poisson brackets of hydrodynamic type in one variable is always semisimple since we can introduce the canonical coordinate as $u = g_2(u) / g_1(u)$. An arbitrary pencil of Poisson brackets of hydrodynamic type $\{ u(x) , u(y) \}_\lambda = \{ u(x) , u(y) \}_2 - \lambda \{ u(x) , u(y) \}_1$ is then of the form
\begin{equation}\label{eq:pencil}
\{ u(x) , u(y) \}_\lambda = 
(u-\lambda) g(u) \delta'(x-y) + \frac12 ((u-\lambda)g(u))' u_x \delta(x-y) .
\end{equation}

We consider the quotient of the space of the dispersive deformations of $\{ u(x) , u(y) \}_\lambda$ modulo the action of the Miura group that preserves the leading term (the so-called pro-unipotent subgroup of the Miura group, that is, the group of the Miura transformations close to identity~\cite{dz01}).

There are two important vector spaces associated with this problem. One of them is the space of  so-called central invariants. One can assign an element of this space to a deformation of the pencil of Poisson brackets. If two deformations are in the same Miura equivalence class, then they have the same central invariant. The other space is the space of obstructions. If this space is equal to zero, then one can show that any infinitesimal dispersive deformation can be extended to a full deformation. Both spaces are instances of the so-called bi-Hamiltonian cohomology associated with the pencil. 

Let us define the bi-Hamiltonian cohomology. The most convenient way to do that is to switch to the $\theta$-formalism, see~\cite{g02,dz01,lz13}. Consider the algebra
$$
\hcA:=C^\infty(\R)[[u^1,u^2,\dots; \theta=\theta^0,\theta^1,\theta^2,\dots]],
$$
where $u^i$ are some formal even variables and $\theta^i$ are some formal odd variables. An element of $C^\infty(\R)$ is a function of the variable $u=u^0$. There are two gradations, $d$ and $p$. The degree  $d$, called standard degree, is defined by assigning $\deg u^i=\deg \theta^i=i$, and the degree $p$ is defined by assigning $\deg u^i=0$, $\deg \theta^i=1$. The subspace of bidegree $(d,p)$ is denoted by $\hcA^p_d$.

By $\d$ we denote the standard derivation on $\hcA$, $\d:=\sum_{s=0}^\infty (u^{s+1}\d_{u^s}+\theta^{s+1}\d_{\theta^s})$. The space of local functional vector fields is the quotient $\hcF=\hcA/\partial\hcA$. Both gradations, $d$ and $p$, descend to the space $\hcF$.

We define two operators on the space $\hcA$ associated to the brackets given in Equation~\eqref{eq:pencil}:
\begin{align*}
D_1& := \sum_{s=0}^\infty \partial^s \left(g\theta^1+\frac{1}{2} \partial_u g u^1 \theta\right) \partial_{u^s} + \partial^s \left(\frac{1}{2} \partial_u g \theta\theta^1\right) \partial_{\theta^s}, \\
D_2& := \sum_{s=0}^\infty \partial^s \left(ug\theta^1+\frac{1}{2} \partial_u (ug) u^1 \theta\right) \partial_{u^s} + \partial^s \left(\frac{1}{2} \partial_u (ug) \theta\theta^1\right) \partial_{\theta^s}.
\end{align*}

These operators commute with $\partial$, so they induce some operators $d_1$ and $d_2$ on the space $\hcF$. The bi-Hamiltonian cohomology in degree $(p,d)$ is the space 
$$
BH^p_d(\hcF,d_1,d_2):=\frac{\ker d_1 \cap \ker d_2 \cap \hcF^p_d} {d_1d_2 \hcF^{p-2}_{d-2}}.
$$

The equivalence classes of dispersive deformations are controlled by the space $BH^2_{\geq 2}(\hcF,d_1,d_2)$. The space of obstructions is the space $BH^3_{\geq 2\delta}(\hcF,d_1,d_2)$, where $\delta$ is the lowest degree $d$ of a non-trivial class in $BH^2_{\geq 2}(\hcF,d_1,d_2)$. Our main result in this paper is Theorem~
\ref{thm:main-theorem}, where we compute completely the bi-Hamiltonian cohomology $BH^p_d(\hcF,d_1,d_2)$.

\subsection{Organization of the paper}

In the computation of the bi-Hamil\-to\-nian cohomology we follow the scheme proposed by Liu and Zhang in~\cite{lz13}. First, we compute the cohomology of $(\hcA[\lambda],D_2-\lambda D_1)$. We perform this computation in Section~\ref{sec:spectralsequence} by introducing a filtration of such complex and showing that the associated spectral sequence converges at the second page. Using that result, in Section~\ref{sec:bi-Hamiltonian}, we obtain $H(\hcF[\lambda],d_2-\lambda d_1)$ by a long exact sequence argument, and, using a lemma by Barakat-Liu-Zhang, we finally compute $BH^p_d(\hcF,d_1,d_2)$. Finally, in Section~\ref{sec:applications}, we discuss some applications of our results.

\section{Main computation} \label{sec:spectralsequence}
In this Section we compute the cohomology of the differential complex $(\hcA[\lambda],D_\lambda)$ with $D_\lambda=D_2-\lambda D_1$, i.e.,
\begin{multline}
D_\lambda = \sum_{s\geq0} \left[ \partial^s\left((u-\lambda) g(u) \theta^1 + \frac12 ( (u-\lambda) g(u) )' u^1 \theta \right) \frac{\partial }{\partial u^s} + \right.\\
 +\left. \frac12 \partial^s \left( ((u-\lambda)g(u))' \theta \theta^1 \right) \frac{\partial }{\partial \theta^s} \right] ,
\end{multline}
where the prime denotes the derivative w.r.t. $u$.

We define a filtration of $\hcA[\lambda]$ which is compatible with $D_\lambda$ and show that the associated spectral sequence collapses at the second page, from which we recover the cohomology.

\subsection{Filtration}

Denote by $\hcA^{(i)}\subset \hcA$ the subspace 
$C^\infty(\R)[[u^1,\dots, u^i, \theta^0,\dots,\theta^i]]$,
and by $\hcA^{[i]}$ the quotient 
$\hcA^{(i)}/\hcA^{(i-1)}$. An upper index $p$, resp. lower
index $d$, denotes, as usual, the homogeneous component of super degree $p$, resp. standard degree $d$.

Let us consider the decreasing filtration of the homogenous degree $d$ component of $\hcA[\lambda]$ defined by 
\begin{equation}
F^i \hcA_d[\lambda] = \hcA_d^{(d-i)}[\lambda], 
\end{equation}
which induces a bounded filtration $F^i \hcA[\lambda]$ of $\hcA[\lambda]$. 

\subsection{Spectral sequence: page zero}

The zeroth page of the associated spectral sequence $E_r^{*,*}$ is 
\begin{equation}
E_0^{p,q} = gr^p \hcA_{p+q}[\lambda] \cong \hcA_{p+q}^{[q]}[\lambda],
\end{equation}
with differential $d_0:E_0^{p,q} \to E_0^{p,q+1}$ given by
\begin{align}
d_0 = & \left( (u-\lambda) g(u) \theta^{q+1} + \frac12 ((u-\lambda) g(u))' u^{q+1} \theta \right) \frac{\partial }{\partial u^q} +
\\ \notag &
+ \frac12 ((u-\lambda) g(u))'\theta \theta^{q+1} \frac{\partial }{\partial \theta^q}  .
\end{align}
Clearly $E_0^{p,q}$ is non-zero only if $p,q \geq0$, i.e., we are dealing with a first-quadrant spectral sequence. Moreover, observe that $E_0^{p,q}$ vanishes for $q=0$, $p\geq1$. 

\subsection{Computing the first page}

By computing the cohomology of $E_0^{*,*}$ we obtain the first page of the spectral sequence. Denote by $\check{C}(\R)$ the space of smooth functions in $u$ quotiented by $\R[u] g^{3/2}(u)$.
\begin{lemma}
The first page of the spectral sequence is given by 
\begin{equation}
E_1^{p,q} = \begin{cases}
\R[\lambda] & p=q=0 \\
\check{C}(\R) \theta \theta^1 & p=0, q=1 \\
\hcA_p^{[q-1]} \theta \theta^q & p\geq1, q\geq2 
\end{cases}
\end{equation}
while all remaining entries are trivial. The differential $d_1:E_1^{p,q} \to E_1^{p+1,q}$ is given by 
\begin{equation}
d_1(f \theta \theta^q) = \left(D_\lambda (f)|_{\lambda=u}  + \frac{q-2}{2} fg \theta^1\right) \theta \theta^q
\in \frac{\hcA_{p+1}^{(q-1)}}{\hcA_{p+1}^{(q-2)}} \theta \theta^q
\end{equation}
for $p\geq1$, $q\geq2$, and zero otherwise.
\end{lemma}
\begin{proof}
For $q=0$ we only have to consider the case $p=0$. Let $f = f_0(u,\lambda) + \theta f_1(u,\lambda)$ be arbitrary element in $E_0^{0,0} = \hcA_0[\lambda]$. Its image under the differential $d_0$ is given by 
\begin{align}\label{eq:image-in-0-1}
d_0(f) = & \left( (u-\lambda) g(u) \theta^{1} + \frac12 ((u-\lambda) g(u))' u^{1} \theta \right) \frac{\partial f_0(u,\lambda)}{\partial u} 
\\ \notag
& +\theta\theta^1 \left(-(u-\lambda)g(u) \frac{\d f_1(u,\lambda)}{\d u}+\frac12 ((u-\lambda)g(u))' f_1(u,\lambda)  \right) .
\end{align}
The equation $d_0 f = 0$ is satisfied iff 
\begin{equation}
\frac{\partial f_0}{\partial u} = 0 
\end{equation}
which implies $f_0 \in \R[\lambda]$, and
\begin{equation}
(u-\lambda)g(u) \frac{\partial f_1}{\partial u} = \frac12 ((u-\lambda)g(u))' f_1 ,
\end{equation}
which does not have nontrivial solutions polynomial in $\lambda$, as one can easily see by substituting $f_1 = (u-\lambda)^s h(u) + O(u-\lambda)^{s+1}$. Hence $E_1^{0,0} = \R[\lambda]$.

For $p\geq0$ and $q\geq1$ one can show by direct computation that
\begin{multline} \label{kern}
\ker ( d_0: E_0^{p,q} \to E_0^{p,q+1} ) =\\ = \left( (u-\lambda) g(u) \theta^q + \frac12 ((u-\lambda)g(u))' u^q \theta \right) \hcA_p^{(q-1)} [\lambda] 
+ \theta \theta^q \hcA_p^{(q-1)} [\lambda] .
\end{multline}
Notice that the case $q=1$ the kernel can be nonzero only for $p=0$, hence $E_1^{p,1}=0$ for $p\geq1$.  

Let us prove~\eqref{kern}. Let $f\in \hcA^{[q]}_{p+q}[\lambda]$ and $d_0 f=0$. As a first step notice that $f$ has to be of the form
\begin{equation}
f = \theta^q f_1 + \theta \theta^q f_2 + \theta u^q f_3 + u^q f_4
\end{equation}
for $f_i \in \hcA^{(q)}_p[\lambda]$ independent of $\theta$ and $\theta^q$. Notice that the coefficient of $u^{q+1} \theta^q$ in $d_0 f =0$ implies that $f_1$ is independent of $u^q$, and, consequently, the coefficient of $u^{q+1}$ implies that $f_4$ is zero. Considering the coefficient of $\theta^{q+1} \theta^q$ in $d_0 f =0$ we get that $f_2$ is independent of $u^q$. Finally the coefficient of $\theta^{q+1} \theta$ gives the equation
\begin{equation}
(u - \lambda) g(u) \frac{\partial ( u^q f_3)}{\partial u^q} = \frac12 ((u-\lambda) g(u))' f_1,
\end{equation}
which clearly implies that $f_3$ is independent of $u^q$, too. We can conclude that $f$ has the form
\begin{equation}
f = \theta^q f_1 + \theta \theta^q f_2 + \theta u^q f_3 
\end{equation}
for $f_i \in \hcA_p^{(q-1)}[\lambda]$ independent of $\theta$. 
Then $d_0f=0$ iff $(u-\lambda) g(u) f_3 = \frac12 ((u-\lambda) g(u))' f_1$. Then $f_1|_{\lambda=u}=0$, hence $f_1 = (u-\lambda) g(u) f_4$ for some $f_4\in \hcA_p^{(q-1)}[\lambda]$, and consequently $f_3=\frac12 ((u-\lambda)g(u))' f_4$. Letting $f_i$ depend on $\theta$ does not change the image, hence~\eqref{kern} is proved. 
%
%
%
%

%
To compute $E_1^{0,1}$ we have to take the quotient the kernel given by Equation~\eqref{kern} for $p=0$, $q=1$ by the image given by~\eqref{eq:image-in-0-1}. It is obvious that choosing $f_1=0$ in~\eqref{eq:image-in-0-1} and an arbitrary $f_0$, we span the first summand in~\eqref{kern}. So, we have to quotient the second summand in~\eqref{kern}, which is in this case $\theta\theta^1 C^\infty(\R)[\lambda]$, by the space of functions 
\begin{equation}\label{eq:0-1-case-essential-image}
\theta\theta^1 \left(-(u-\lambda)gh'+\frac12 ((u-\lambda)g)' h  \right),
\quad h\in C^\infty(\R)[\lambda].
\end{equation}
In order to do that, we substitute $h=g^{1/2}t$. This turns~\eqref{eq:0-1-case-essential-image} into
\begin{equation}
\label{eq:0-1-case-simplified-image}
\theta\theta^1 g^{\frac 32} \left(-(u-\lambda)t'+\frac12 t  \right),
\quad t\in C^\infty(\R)[\lambda].
\end{equation} 
The quotient of $\theta\theta^1 C^\infty(\R)[\lambda]$ modulo this space is equal to the quotient of $\theta\theta^1 C^\infty(\R)$ modulo the space of functions of the form~\eqref{eq:0-1-case-simplified-image} that do not depend on $\lambda$. Let us determine this space.

Since $t$ is a polynomial in $\lambda$, we can represent it as $t=\sum_{i=0}^n (u-\lambda)^i t_i$, $t_i\in C^\infty(\R)$. The condition that $-(u-\lambda)t'+t/2$ does not depend on $\lambda$ is equivalent to the set of equations
\begin{equation}\label{eq:condition-independent-lambda}
t_i' = \left(i+\frac 12\right) t_{i+1},\quad i=0,1,\dots,n,
\end{equation}
where we assume $t_{n+1}=0$. If this condition is satisfied, then $-(u-\lambda)t'+t/2=t_0/2$, and~\eqref{eq:condition-independent-lambda} implies that $t_0^{(n+1)}=0$. This means that $t_0/2$ is a polynomial in $u$. So, the space of functions of the form~\eqref{eq:0-1-case-simplified-image} that do not depend on $\lambda$ is equal to $\theta\theta^1 g^{3/2} \R[u]$.  This implies that 
$E^{0,1}_1 = \check C(\R)\theta\theta^1$.


For $p\geq0$ and $q\geq2$ the image of $d_0 : E_0^{p,q-1} \to E_0^{p,q}$ is given by elements of the form
\begin{multline} \label{im}
\left( (u-\lambda) g(u) \theta^q + \frac12 ((u-\lambda)g(u))' u^q \theta \right) \frac{\partial h_0}{\partial u^{q-1}}  \\
+ \theta^q \theta \left( (u-\lambda) g(u) \frac{\partial h_1}{\partial u^{q-1}} - \frac12 ((u-\lambda)g(u))' \frac{\partial h_0}{\partial \theta^{q-1}} \right)   
\end{multline}
for $h_0, h_1 \in \hcA_{p+q-1}^{[q-1]}[\lambda]$, independent of $\theta$.

For $p=0$ and $q\geq2$ the two arbitrary elements of $\hcA_0[\lambda]$ appearing in~\eqref{kern} can be killed by~\eqref{im},  using the fact that in this case both $h_0$ and $h_1$ are linear in $u^{q-1}$, $\theta^{q-1}$. This gives $E_1^{0,q}=0$ for $q\geq2$.

Finally we compute $E_1^{p,q}$ for $p\geq1$, $q\geq2$. Observe that choosing an appropriate $h_0$ we can use the first summand in~\eqref{im} kill completely the first summand in~\eqref{kern}. The choice of $h_0$ has some freedom, namely, it is fixed up to $u^{q-1}$-independent term $\tilde h_0$.  Now, the second summand in~\eqref{im} can be rewritten as 
\begin{equation} \label{eq:im-secondsummand}
 \theta^q \theta \left( (u-\lambda) g \frac{\partial h_1}{\partial u^{q-1}} - (u-\lambda) \frac {g'}2  \frac{\partial h_0}{\partial \theta^{q-1}} + \frac{1}{2}g \frac{\d h_0}{\d \theta^{q-1}}\right).   
\end{equation}
The constant term of this expression in $(u-\lambda)$ is given by the third summand that is fixed up to $\frac g2 \frac{\d \tilde h_0}{\d \theta^{q-1}}$, where $\frac{\d \tilde h_0}{\d \theta^{q-1}}$ (and, therefore, $\frac g2 \frac{\d \tilde h_0}{\d \theta^{q-1}}$) is an arbitrary element in $\hcA_{p}^{(q-2)}$. All terms proportional to $(u-\lambda)$ in the second summand in~\eqref{kern} (corrected by the second summand in~\eqref{eq:im-secondsummand} that depends on the choice of $h_0$ and $\tilde h_0$) can be killed by the choice of $h_1$ .  This means that the quotient of the kernel~\eqref{kern} by the image~\eqref{im} for  $p\geq1$, $q\geq2$ is equal to 
\begin{equation}\label{eq:quotient}
\theta\theta^q\frac{\hcA_p^{(q-1)}[\lambda]}{(u-\lambda)\hcA_p^{(q-1)}[\lambda]+\hcA_p^{(q-2)}}.
\end{equation}

The differential $d_1$ is given by the restriction to $\hcA_p^{[q-1]}\theta\theta^q$ of the part of the differential $D_\lambda$ that sends $\hcA_d^{[d-i]}[\lambda]$ to $\hcA_{d+1}^{[d-i]}[\lambda]$, $d=p+q$, $d-i=q$, with a subsequent corestriction to the quotient~\eqref{eq:quotient}. In general, we can represent $D_\lambda(f)$, $f\in \hcA_d^{[d-i]}[\lambda]$, as a sum of three summands living in $\hcA_{d+1}^{[d-i+1]}[\lambda]$, $\hcA_{d+1}^{[d-i]}[\lambda]$, and $\hcA_{d+1}^{(d-i-1)}[\lambda]$. The first component here is exactly $d_0$, so it vanishes on the kernel of $d_0$; the quotient~\eqref{eq:quotient} implies that we disregard the third component and substitute $\lambda=u$ in the second component of that expression. Then it is straightforward to show that the second component of $D_\lambda$ applied to the representative $f\theta\theta^q$ is equal to $(D_\lambda(f)|_{\lambda=u}+\frac{q-2}{2} g\theta^1 f)\theta\theta^q$. We refer to~\cite[Lemma 7]{cps14} for a more detailed explanation.
\end{proof}


\subsection{The perturbation lemma and the second page}

We now proceed to compute the second page of the spectral sequence. In the KdV case~\cite{cps14} an explicit homotopy contraction was found, which allowed us to show that $d_1$ was acyclic everywhere except for a few explicitly computed cases. Here we will use the homological perturbation lemma to construct a homotopy contraction for $d_1$.

The following lemma is a well-known result in homological perturbation theory,  see for example~\cite{bl91}. 
\begin{lemma}
Let $(C,d)$ be a differential complex with a differential $d\colon C^k \to C^{k+1}$ given by $d = d_0 + \alpha$ where $d_0$ is also a differential. Let $h_0\colon C^k \to C^{k-1}$ be a homotopy contraction for $d_0$, i.e., $h_0 d_0 + d_0 h_0 = 1$. The formal map
\begin{equation} \label{homo}
h := \sum_{n\geq0} (-1)^n (h_0 \alpha)^n h_0,
\end{equation} 
if well-defined, is a homotopy contraction for $d$. 
\end{lemma}

Let us use it to compute the second page of our spectral sequence.

\begin{lemma}
The second page of the spectral sequence is given by
\begin{equation}
 E_2^{p,q} = \begin{cases}
\R[\lambda] & p=0, q=0, \\
\check{C}(\R) \theta \theta^1 & p=0, q=1 \\
C^\infty (\R) \theta \theta^1 \theta^2 & p=1, q=2 \\
0 &\text{else.}
\end{cases}
\end{equation}
The differential $d_2$ vanishes.  
\end{lemma}

\begin{proof}
Assume $p\geq1$, $q\geq2$. Let us split the differential $d_1:E_1^{p,q} \to E_1^{p+1,q}$ as $d_1 = \theta^1 U + \theta^1 V + W$ where $U$, $V$ and $W$ do not contain $\theta^1$ and 
\begin{equation}
U = g(u) \left( \sum_{s=1}^{q-1} \frac{s+2}2 u^s \frac{\partial }{\partial u^s} + 
\sum_{s=0}^q \frac{s-1}2  \theta^s \frac{\partial }{\partial \theta^s} \right)  
\end{equation}
(this determines $V$ and $W$ uniquely).
Clearly $g^{-1}(u) U$ is the operator that measures the degree of a monomial in $\hcA$ where the weight of each $u^s$ is $(s+2)/2$ and of $\theta^s$ is $(s-1)/2$. Notice that a degree zero monomial in $\hcA^{(q-1)}_{p} \theta \theta^q$  is possible only for $p=1$, $q=2$ and is of the form $f(u) \theta^1 \theta \theta^2$.

Since $\theta^1 U$ squares to zero, it is a differential, and we consider $d_1$ as its deformation. 
For $(p,q) \not= (1,2)$ the operator $U$ is invertible and $U^{-1} \frac{\partial }{\partial \theta^1}$ is an homotopy contraction for $\theta^1 U$. Setting $h_0 = U^{-1} \frac{\partial }{\partial \theta^1}$ and $\alpha = \theta^1 V + W$, thanks to the homotopy perturbation lemma we can conclude that formula~\eqref{homo} gives an homotopy contraction for $d_1$, as long as it is well-defined. Explicitly we have
\begin{equation}
h = \sum_{n\geq0} (-1)^n \left( U^{-1} \frac{\partial }{\partial \theta^1} ( \theta^1 V + W )  \right)^n U^{-1} \frac{\partial }{\partial \theta^1} . 
\end{equation}
In this formula we can remove $W$, since it will always appear in the combination $\frac{\partial }{\partial \theta^1} W U^{-1} \frac{\partial }{\partial \theta^1}$ which clearly vanishes. Moreover the derivative w.r.t. $\theta^1$ inside the brackets can only act on the $\theta^1$ in front of $V$. Hence we have
\begin{equation}
h = \sum_{n\geq0} (-1)^n \left( U^{-1} V \right)^{n} U^{-1} \frac{\partial }{\partial \theta^1} .
\end{equation}
To show that this is well-defined we need first to show that only a finite number of terms in $h$ are non-zero when it acts on a monomial. A simple computation shows that
\begin{multline}
V = \sum_{s=2}^{q-1} \sum_{l=1}^{s-1}  \frac{s+2}2\binom{s}{l}  (\partial^l g) u^{s-l}\frac{\partial }{\partial u^s} +\\
+ \sum_{s=1}^{q-1} \sum_{l=0}^{s-1} \frac{l-1}2  \binom{s}{l} \left(\partial^{s-l} \partial_u ((u-\lambda) g(u)) \right)_{|\lambda=u} \theta^l\frac{\partial }{\partial \theta^s} .
\end{multline}

Let us introduce a lexicographic order on the monomials in $\hcA$ by writing each monomial $\mathfrak{m}$ in the standard form
\begin{equation}
\cdots (\theta^k)^{j_k} (u^k)^{i_k} (\theta^{k-1})^{j_{k-1}} (u^{k-1})^{i_{k-1}} \cdots (\theta^1)^{j_1} (u^1)^{i_1} (\theta)^{j_0} f(u)
\end{equation}
with which we associate the multiindex $I(\mathfrak{m}) = ( \cdots, j_1, i_1, j_0)$, 
with $j_k=0,1$ and $i_k=0,1,2,\dots$. We say that $\mathfrak{m}$ is of higher lexicographic order than $\mathfrak{n}$ iff the leftmost non-zero entry of $I(\mathfrak{m}) - I(\mathfrak{n})$ is positive.

Clearly $V$ strictly decreases such lexicographic order and preserves the standard degree. Since there are only a finite number of monomials of fixed standard degree and with lexicographic order lower than that of a fixed monomial, we conclude that, when $h$ is applied to a monomial, the sum in its definition is finite. 

Finally we need to check that the inverse of $U$ in $h$ is always well-defined, that is to say that $U^{-1}$ is never acting on monomials of degree zero. This follows simply from the observation that when we apply $h$ to a monomial of strictly positive degree, the operators $\frac{\partial }{\partial \theta^1}$ and $U^{-1}$ leave its degree invariant, while $V$ can only increase the degree. We conclude that $h$ is a well-defined homotopy contraction for $d_1$ if $(p,q)\not= (1,2)$.

This allows us to conclude that $d_1$ is acyclic on $E_1^{p,q}$ for $p\geq1$, $q\geq2$, unless $(p,q) = (1,2)$. 

In the case $(p,q)=(1,2)$, the space $E_1^{1,2}=C^\infty(\R)u^1\theta\theta^2\oplus C^\infty(\R)\theta^1\theta\theta^2$. The first summand here is a part of the argument above on the homotopy contraction, while the second one lies in the kernel of $d_1$. Thus we conclude that the cohomology $E_2^{1,2}$ is equal to $C^\infty(\R)\theta^1\theta\theta^2$.

The differential $d_2\colon E_2^{p,q}\to E_2^{p+2,q-1}$ vanishes for gradation reasons: there are no pairs $(p,q)$ such that both  $E_2^{p,q}$ and $E_2^{p+2,q-1}$
are non-trivial.
\end{proof}

\subsection{The cohomology of $(\hcA[\lambda],D_\lambda)$}

For the same reason as $d_2$, all differentials on the higher pages of the spectral sequence vanish, hence it converges. 
We can then reconstruct the cohomology of the polynomial complex $(\hcA[\lambda],d_\lambda)$ from the limit term $E_\infty = E_2$ of the spectral sequence. 
\begin{proposition}
	We have:
\begin{equation} \label{lamcohom}
H(\hcA[\lambda],D_\lambda) = \R[\lambda]\oplus \check{C}(\R)\theta\theta^1 \oplus C^\infty (\R) \theta\theta^1\theta^2 .
\end{equation}	
\end{proposition}

\section{The bi-Hamiltonian cohomology}\label{sec:bi-Hamiltonian}

In this Section we use the computation of $H(\hcA[\lambda],D_\lambda)$ from the previous Section in order to compute the bi-Hamiltonian cohomology $BH^p_d(\hcF, d_1, d_2)$. 

First, we compute the cohomology of $(\hcF[\lambda],d_\lambda)$. 
\begin{proposition}
	The cohomology groups $H^p_d(\hcF[\lambda],d_\lambda)$ are isomorphic to
\begin{equation}
H^p_d(\hcF[\lambda],d_\lambda) \cong \begin{cases}
C^\infty (\R) \slash \R[u]& (p,d) = (1,1), (2,1), \\
C^\infty (\R) &(p,d) = (2,3), (3,3), \\
\R[\lambda] & (p,d) = (0,0) , \\
0 & \text{else}.
\end{cases}
\end{equation}

\end{proposition}

\begin{proof}
The cohomology of $(\hcF[\lambda],d_\lambda)$ is related to the cohomology of $(\hcA[\lambda],d_\lambda)$ by the long exact sequence~\cite[Lemma 3.7]{lz13}:
\begin{equation}
H^p_{d-1}(\hcA[\lambda]/\R[\lambda])\to H^p_{d}(\hcA[\lambda])
\to H^p_d(\hcF[\lambda])\to H^{p+1}_d(\hcA[\lambda]/\R[\lambda]).
\end{equation}
Here $H^p_d(\hcA[\lambda]/\R[\lambda])=H^p_d(\hcA[\lambda])$ in all cases except for $p=d=0$. Since almost all cohomology groups of $\hcA[\lambda]$ and $\hcA[\lambda]/\R[\lambda]$ are equal to zero, this long exact sequence allows to compute all cohomology groups of $\hcF[\lambda]$. We refer to the proof of Proposition 13 in~\cite{cps14} for details. 
\end{proof}

\begin{theorem} \label{thm:main-theorem}
The bi-Hamiltonian cohomology groups $BH(\hcF,d_1,d_2)$ are given by
\begin{equation}
BH^p_d(\hcF,d_1,d_2) \cong \begin{cases}
C^\infty (\R) & (p,d) = (1,1), (2,1), (2,3), (3,3) , \\
\R & (p,d) = (0,0) , \\
0 & \text{else}.
\end{cases}
\end{equation}
\end{theorem}
\begin{proof}
Lemma~4.4 in~\cite{lz13} identifies $BH^p_d(\hcF,d_1,d_2)$  with $H(\hcF[\lambda],d_\lambda)$ unless  $$(p,d) = (0,0), (1,0), (1,1), (2,1).$$ These four exceptional cases can be obtained by a simple explicit computation. 
\end{proof}

		
		

\section{Some remarks}\label{sec:applications}

Let us now formulate the main deformation theorem, which generalizes Theorem~1.2 in~\cite{lz13}. 

\begin{theorem}
For any  function $c(u)\in C^\infty(\R)$ there is a unique equivalence class of deformations of the Poisson pencil $(u-\lambda) g\theta\theta^1$ with a representative given by
\begin{multline}\label{eq:deformation-in-theta}
	(u-\lambda )g\theta\theta^1 + \frac{\epsilon^2}2 \left[ 6cg^2\theta\theta^3+(9c g g' u^1+ 6 c' g^2 u^1)\theta\theta^2 + \right.\\ 
	\left. + (-5c(g')^2 (u^1)^2 + c' g g'(u^1)^2+4c g g''(u^1)^2 +5c g g'u^2)\theta\theta^1\right]+ O(\epsilon^4).
	\end{multline} 
\end{theorem}

\begin{proof} This follows from the fact that the only non-trivial bi-Hamiltonian cohomology $H^2_d(\hcF,d_1,d_2)$ is $H^2_3(\hcF,d_1,d_2)=C^\infty(\R)\ni c$, and since for all $d\geq 4$ $H^3_{d} (\hcF,d_1,d_2)=0$, the first order deformation corresponding to $c$ is unobstructed. In order to compute $\epsilon^2$ term in the deformation explicitly, we have to apply the following formula from~\cite{dlz06}: 
	\begin{equation}\label{eq:DLZ-element}
	d_1\left(d_2 [cu^1\log u^1] -d_1 [ucu^1\log u^1]\right).
	\end{equation}
Here by $[\cdot]$ we denote the class of an expression in (some extension of) $\hcF$. Obviously, in this computation we extend a little bit the space of allowed functions with logarithm and rational functions, but they disappear in the final formula (and this is the reason why we still get a non-trivial class in the bi-Hamiltonian cohomology of $\hcF$ --- otherwise it would be a $d_1d_2$-exact element). Explicit computation of~\eqref{eq:DLZ-element} gives the $\epsilon^2$-term in~\eqref{eq:deformation-in-theta}. 
\end{proof}

An equivalent formulation can be given in the $\delta$-formalism (the correspondence between the $\theta$-formalism and the $\delta$-functions formalism for the Poisson brackets is explained in~\cite[Section 3]{lz13}): for any choice of metric $g(u)$ and of central invariant $c(u)$ there exists a unique equivalence class of dispersive Poisson pencils with a representative that at the first two orders in $\epsilon$ is equal to
\begin{align}
\label{can-form}
\{ u(x), u(y) \}_1 &= g(u) \delta'(x-y) , \\
\{ u(x), u(y) \}_2 &=  u g(u) \delta'(x-y) + \frac12 (u g(u))' u^1 \delta(x-y) + \\
&+ \epsilon^2 \Big[  3 c g^2 \delta'''(x-y) 
+ \left( \frac92 g^2 c' u^1 + 9 g g' c u^2 \right) \delta''(x-y) +\\
&\qquad + P_{2,1} \delta'(x-y) + P_{2,0} \delta(x-y) \Big] + O(\epsilon^4)
\end{align}
where 
\begin{equation}
P_{2,1} = \left( 
   8 g g' c' + 2(g')^2 c + \frac{13}{2} g g'' c + \frac32 g^2 c''
    \right) (u^1)^2 
  + \left(  \frac32 g^2 c' + 7 g g' c \right) u^2
  \end{equation}
and
\begin{align}
P_{2,0} &=  \left( \frac12 (g')^2 c' + gg'c''+ \frac{11}4 gg'' c' + \frac34 g'g'' c  +\frac74 gg''' c \right) (u^1)^3 + \\
&+ \left( 4 gg'c' + (g')^2 c + \frac{11}2 g g'' c \right) u^1 u^2 +2gg'c u^3 .
\end{align}
Here the prime denotes derivation w.r.t. the variable $u$, but when applied on the delta functions, where it denotes derivation w.r.t. $x$. 
A similar formula appears in~\cite{al11, d06} but the flat coordinate of the metric of the first dispersionless Poisson bracket is used, rather than the canonical coordinate $u$.

Notice that, while a choice of central invariant, i.e. of infinitesimal deformation, allows to reconstruct a Poisson pencil in all orders in $\epsilon$,  such full dispersive form is known explicitly only in a few important examples that we recall below.

\subsubsection*{KdV}
The KdV hierarchy admits the well-known bi-Hamiltonian structure
\begin{align}
&\{ u(x), u(y) \}_1 = \delta'(x-y) , \\
&\{ u(x), u(y) \}_2 = u(x) \delta'(x-y) + \frac12 u'(x) \delta(x-y) +\frac{\epsilon^2}8 \delta'''(x-y)  ,
\end{align}
where $u$ is indeed the canonical coordinate, i.e. $g(u)=1$.
The central invariant is also constant:
\begin{equation}
c(u) = \frac1{24} .
\end{equation}
Recall that the central invariant of a scalar Poisson pencil is given by
\begin{equation}
\label{central}
c(u) = \frac1{3g^2} (Q_2 - u Q_1),
\end{equation}
where $Q_a$ is the coefficient of $\epsilon^2 \delta'''(x-y)$ in the Poisson bracket $\{u(x),u(y)\}_a$, and $u$ the associated canonical coordinate.

\subsubsection*{Camassa-Holm} The Camassa-Holm~\cite{ch93} hierarchy can be defined by means of the bi-Hamiltonian structure given by the 
pair of Poisson structures
\begin{align}
&\{ w(x), w(y) \}_1 = \delta'(x-y) -\frac{\epsilon^2}8 \delta'''(x-y), \\
&\{ w(x), w(y) \}_2 = w(x) \delta'(x-y) + \frac12 w'(x) \delta(x-y).
\end{align}
From equation \eqref{central} the central invariant is easily computed to be equal to $c(w)=w/24$. The Poisson pencil above
is not quite in the canonical form \eqref{can-form}. To write it in canonical form, we apply the second type Miura transformation
$w=u+\frac{\epsilon}{2\sqrt{2}}u_{x}$. In this new coordinate $u(x)$, the Poisson pencil is computed to be 
\begin{align}
&\{ u(x), u(y) \}_1 = \delta'(x-y), \\
&\{ u(x), u(y) \}_2 = u(x) \delta'(x-y) + \frac12 u'(x) \delta(x-y) \\ &\hspace{2.5cm}+\frac{\epsilon^{2}}{8}\left(u\delta'''(x-y)+\frac{3}{2}u'(x)\delta''(x-y)+\frac{1}{2}u''(x)\delta'(x-y)\right)\\ 
&\hspace{4cm} +O(\epsilon^{4}).
\end{align}
From the term proportional to $\epsilon^{2}$ we can, again, using the formula \eqref{central}  compute the central invariant to agree with the 
previous value.
\subsubsection*{Volterra system}
The bi-Hamiltonian structure of the Volterra hierarchy~\cite{dlyz14} is
\begin{align}
\{ v(x), v(y) \}_1 &= \frac1\epsilon \left[ \delta(x-y+\epsilon)  - \delta(x-y-\epsilon) \right] , \\
\{ v(x) , v(y) \}_2 &= \frac1{4\epsilon}  \big[ (e^{v(x)} +e^{v(y)} )(\delta(x-y+\epsilon) - \delta(x-y-\epsilon) ) + \notag \\
&\qquad \qquad+ v(x+\epsilon) \delta(x-y+2\epsilon) - v(y+\epsilon) \delta(x-y-2\epsilon) \big] .
\end{align}
The variable $v$ is the flat coordinate for the metric associated to the dispersionless limit of the first Poisson bracket. Rewriting the Poisson brackets in the canonical coordinate $u= 4e^v$, we obtain: 
\begin{align} \label{voltcan}
\{ u(x), u(y) \}_1 &= \frac1\epsilon u(x) u(y) \left[ \delta(x-y+\epsilon)  - \delta(x-y-\epsilon) \right] , \\
\{ u(x) , u(y) \}_2 &= \frac1{4\epsilon} u(x) u(y) \big[ (u(x) +u(y) )(\delta(x-y+\epsilon) - \delta(x-y-\epsilon) ) + \notag \\
&\qquad+ u(x+\epsilon) \delta(x-y+2\epsilon) - u(y+\epsilon) \delta(x-y-2\epsilon) \big] . \label{voltcan1}
\end{align}

The dispersionless limit of the pencil $\{, \}_\lambda = \{,\}_2 - \lambda \{, \}_1$ is 
\begin{equation}
 \{ u(x), u(y) \}^{[0]}_\lambda = ( u(x) - \lambda) 2 u(x)^2 \delta'(x-y) +
 \frac12 (  ( u(x) - \lambda) 2 u(x)^2 )' \delta(x-y) ,
\end{equation}
i.e., it corresponds to the flat pencil of contravariant metrics 
\begin{equation}
2 u^3 -  2 \lambda  u^2  .
\end{equation}
Notice that  $u$ is indeed the canonical coordinate, so in this case $g(u)=2 u^2$. We can  read the coefficients $Q_1=u^2/3$, $Q_2=5u^3/6$ of $\epsilon^2 \delta'''$ from~\eqref{voltcan}-\eqref{voltcan1}, so the central invariant is 
\begin{equation}
c(u)  = \frac1{24u} .
\end{equation}

\section*{Acknowledgments}
The authors would like to thank B. Dubrovin for his interest in this work.
This work was supported by the Netherlands Organization for Scientific Research.

\appendix

\end{document}